\documentclass[11pt]{article}
\usepackage{amsmath, fullpage, mathtools}
\usepackage{amscd}
\usepackage{amsthm} \usepackage{amssymb}
\usepackage{latexsym}
\usepackage{eufrak}
\usepackage{euscript}
\usepackage{epsfig}
\usepackage{graphics}
\usepackage{array}
\usepackage{enumerate} 
\theoremstyle{theorem}
\newtheorem{theorem}{Theorem}[section]
\theoremstyle{corollary}
\newtheorem{corollary}{Corollary}[section]
\theoremstyle{lemma}

\theoremstyle{definition}

\theoremstyle{proof}
\theoremstyle{remark}
\newtheorem{remark}{Remark}[section]
\theoremstyle{example}

\theoremstyle{observation}

\usepackage{authblk}
\begin{document}

\setcounter{Maxaffil}{2}
\title{Eigenvalue bounds for some classes of  matrices associated with graphs}
\author[a]{\rm Ranjit Mehatari\thanks{ranjitmehatari@gmail.com, mehatarir@nitrkl.ac.in}}
\author[b]{\rm M. Rajesh Kannan\thanks{rajeshkannan1.m@gmail.com, rajeshkannan@maths.iitkgp.ac.in}}
\affil[a]{Department of Mathematics,}
\affil[ ]{National Institute of Technology Rourkela,}
\affil[ ]{Rourkela - 769008, India}
\affil[ ]{ }
\affil[b]{Department of Mathematics,}
\affil[ ]{Indian Institute of Technology Kharagpur,}
\affil[ ]{Kharagpur-721302, India.}
\maketitle

\maketitle
\begin{abstract}

For a given complex square matrix $A$ with constant row sum, we establish two new eigenvalue inclusion sets. Using these bounds, first we derive bounds for the second largest and smallest eigenvalues of adjacency matrices of  $k$-regular graphs. Then, we establish some bounds for the second largest and the smallest eigenvalues of the normalized adjacency matrices of graphs and the second smallest eigenvalue and the largest eigenvalue of the Laplacian matrices of graphs. Sharpness of these bounds are verified by examples.

\end{abstract}
\section{Introduction}

In this paper we consider simple, connected, finite and undirected graphs. Let $G=(V,E)$ be a graph with the vertex set $V=\{1,2,\ldots, n\}$ and the edge set $E$. If two vertices $i$ and $j$ of $G$ are adjacent, we denote it by $i\sim j$. For each vertex $i$, let $d_i$ denote the degree of the vertex $i$. A graph is said to be $d$-regular, if $d_i = d$ for all $i$.   The average degree of a graph $G$, denoted by $\Delta$, defined as $\frac{\sum_{i=1}^{n}d_i}{n}$. A vertex $i$ is said to be a dominating vertex if it is adjacent to all other vertices, i.e., $d_i=n-1$. We use $N(i,j)$ to denote the number of common neighbors of the vertices $i,j\in V, i\neq j$. For more of graph theoretic terminology, we refer to \cite{Bond}.

The  $(0,1)$-adjacency matrix  $A=[a_{ij}]$ of a  graph $G$, on $n$ vertices, is an $n \times n$ matrix defined by
$$a_{ij}=\begin{cases}
1,& if\ i\sim j,\\
0,& otherwise.
\end{cases}$$
Let $D$ denote the diagonal matrix whose $(i,i)^{th}$  entry is $d_i$. Then the matrix $L=D-A$ is called the Laplacian matrix of the graph $G$, and the matrix $\mathcal{A}=D^{-1}A$ is called the normalized adjacency matrix (or the transition matrix) of $G$ . For more details we refer to \cite{Brou, Chung, Cve1}. The matrix $\mathcal{A}$ is similar to the Randi\'c matrix \cite{Boz}. For any real number $\alpha$, the general Randi\'c index $R_\alpha(G)$\cite{Ran, Boll} of $G$ is defined by $$R_\alpha(G)=\sum_{\substack{i\sim j}}d_i^\alpha d_j^\alpha.$$
All the matrices $A,\ L$ and $ \mathcal{A}$ have real eigenvalues and reflect various interesting properties of the underlying graph $G$. Each of them comes with a set of strengths and weaknesses, which are discussed elaborately in \cite{But}.
 Various bounds for  the normalized adjacency matrix of a graph can be found in  \cite{Ban1,Chung,LiGu,Rojo}.   In \cite{Stan}, the author surveys eigenvalue bounds for different types of matrices associated with a graph viz., adjacency matrix, Laplacian matrix, signless Laplacian matrix, etc. In this paper, we provide some new bounds for the eigenvalues of  adjacency matrix, normalized adjacency matrix and Laplacian matrix including the algebraic connectivity(second smallest eigenvalue \cite{Fed}).

Let $A=[a_{ij}]$ be an $n \times n$ complex square matrix. The $i$-deleted absolute row sum of $A$ is defined by  $$r_i(A)=\sum_{j\in I\setminus\{i\}}|a_{ij}|,\ \forall i\in I,$$
where $I=\{1,2,\ldots,n\}$. Ger{\v s}gorin(1931) proved the following result. It is now well known as the Ger{\v s}gorin circle theorem.

\begin{theorem}
Let $A=[a_{ij}]$ be an $n\times n$ complex matrix. Then the eigenvalues of $A$ lie in the region
$$G_A=  \bigcup_{i=1}^n \Gamma_i(A),$$  where
$$ \Gamma_i(A)= \{z\in\mathbb{C}:|z-a_{ii}|\leq r_i(A)=\sum_{j\neq i}|a_{ij}|\}.$$
\end{theorem}
The set $\Gamma_i(A)$ is called the  $i^{th}$ Ger{\v s}gorin disk of $A$, and $G_A$ is called the Ger{\v s}gorin region of $A$.

In the same vein, Ostrowski (1937) and Brauer (1947) proved the following result.
\begin{theorem}
    \label{Bau}
        Let $A=[a_{ij}]$ be an $n\times n$ complex matrix. Then the eigenvalues of $A$ lie in the region
    $$K_A=\bigcup_{\substack{i,j\in I\\ i\neq j}}K_{ij}(A),$$
    where for $i\neq j$, $K_{ij}(A)$ is defined by
    $$K_{ij}(A)=\{z\in\mathbb{C}:|z-a_{ii}||z-a_{jj}|\leq r_i(A)r_j(A)\}.$$
\end{theorem}
The set $K_{ij}(A)$ is called the $(i, j)^{th}$-Brauer-Cassini oval for the matrix A and $K_A$ is called the Brauer region. For more details about the localization of eigenvalues, we refer to \cite{Varga}.
In \cite{Ban1},  the following  property of irreducible row-stochastic matrices established:
\begin{theorem}
    \label{sub1}
    Let $A=[a_{ij}]$ be an irreducible row-stochastic square matrix of order n. Then  all eigenvalues of $A$ other than 1 are eigenvalues of
    $$A(k)=A(k|k)-j_{n-1}a(k),\ \ k=1,2,\ldots, n,$$
    where $A(k|k)$ is the $k^{th}$ principal submatrix, $a(k)^T=\left[\begin{array}{cccccc}
    a_{k1}&\cdots&a_{k,k-1}&a_{k,k+1}&\cdots&a_{kn}
    \end{array}\right]$ is the $k$-deleted row of $A$ and $j_{n-1}$ is the $n-1$ component row vector with all entries equal to 1.
\end{theorem}
 In \cite{Hall}, the authors obtained results similar to the above theorem for the real matrices. In this paper we extend this result for any complex square matrix $A$ with constant row sum. If $A$ is a  complex square matrix with the constant row sum $\gamma$, then $\gamma$ is always an eigenvalue of $A$. We construct $n$ number of block upper diagonal complex square matrices each of them is similar to $A$. Then, we apply Ger{\v s}gorin and Brauer theorems to these $n$ matrices to get some new eigenvalue inclusion regions.

Let $A$ be an $n \times n$ complex matrix with real eigenvalues and the eigenvalues are ordered as follows: $\lambda_1 \geq \lambda_2 \geq \dots \geq \lambda_n.$
Next, we  recall the eigenvalue inequalities for this class of matrices.
\begin{theorem}\cite{Wolk}
\label{tr1}
Let B be an $n\times n$ complex matrix with real eigenvalues
and let $$m=\frac{\text{trace } B}{n}\ \ \text{ and }\ \ s^2=\frac{\text{trace } B^2}{n}-m^2, $$
then
\begin{equation}
\label{tr_eqn1}
m-s(n-1)^\frac{1}{2}\leq\lambda_n\leq m-s/(n-1)^\frac{1}{2},
\end{equation}
\begin{equation}
\label{tr_eqn2}
m+s/(n-1)^\frac{1}{2}\leq\lambda_1\leq m+s(n-1)^\frac{1}{2}.
\end{equation}
Equality holds on the left (right) of (\ref{tr_eqn1}) if and only if equality holds on the
left (right) of (\ref{tr_eqn2}) if and only if the $n-1$ largest (smallest) eigenvalues are
equal.
\end{theorem}
Using some of the  extensions of Theorem \ref{sub1} and  Theorem \ref{tr1},  we obtain bounds for the eigenvalues of adjacency matrix(of $d$-regular graph), normalized adjacency matrix and Laplacian matrix. Tightness of these bounds are illustrated with examples. \\

This article is organized as follows: In Section \ref{Sec2}, we establish some eigenvalue localizing theorems for a complex square matrix with constant row sum. In Section \ref{Sec3}, we derive bounds for the second largest eigenvalue and the least eigenvalue of the adjacency matrix of a regular graph. In Section \ref{Sec4}, we establish some bounds for the second largest eigenvalue and the least eigenvalue of the normalized adjacency matrix of any connected graph. Finally, we provide bounds for the second smallest eigenvalue(algebraic connectivity) and largest eigenvalue of the Laplacian matrix of any connected graph. This is done in Section \ref{Sec5}.

\section{Some eigenvalue inclusion sets for complex matrices with constant row sum}\label{Sec2}
For an $n\times n$ complex matrix $A$, let $A(k|k)$ denote the $k^{th}$ principal submatrix of $A$ obtained by deleting the $k^{th}$ row and the $k^{th}$ column of $A$. Let  $j_{n}$ denote row vector of size $n$, with all entries equal to $1$ and $e$ denote the column vector of appropriate size with all entries are $1$.
\begin{theorem}
    \label{sub2}
    Let $A=[a_{ij}]$ be an $n\times n$ complex  matrix with the constant row sum $\gamma$. Then $\gamma$ is an eigenvalue of $A$. Furthermore, $A$ is similar to the complex block upper triangular matrix $$\left[\begin{array}{cc}
        \gamma&a(k)^T\\
        0&A(k)\end{array}\right],$$
        where $A(k)=A(k|k)-j_{n-1}a(k),\ \ k=1,2,\ldots, n,$ and
 $a(k)^T=\left[\begin{array}{cccccc}
    a_{k1}&\cdots&a_{k,k-1}&a_{k,k+1}&\cdots&a_{kn}
    \end{array}\right]$ is the $k$-deleted row of $A$.
\end{theorem}
\begin{proof}
Let $e_i$ denote the column vector with  $1$ at the $i^{th}$ position and $0$ elsewhere.
Let $P_1=I_n$ and for $k>1$, consider the permutation matrix $P_k=\left[\begin{array}{cccccccc}
e_2&e_3&\cdots&e_{k}&e_1&e_{k+1}&\cdots&e_n
\end{array}\right]$. \\
Therefore, the matrix $A$ is similar to the matrix
\begin{eqnarray*}
    A_k=P_k^{-1}AP_k=\left[\begin{array}{cc}
        a_{kk}&a(k)^T\\
        y&A(k|k)\end{array}\right],\ \forall k=1,2,\ldots,n,
\end{eqnarray*}
where $y=\left[\begin{array}{ccccccc}
a_{1k}&a_{2k}&\cdots&a_{k-1,k}&a_{k+1,k}&\cdots&a_{nk}
\end{array}\right]^T.$\\
Let $P=\left[\begin{array}{cccccccc}
e&e_2&e_3&\cdots&e_{k}&e_{k+1}&\cdots&e_n
\end{array}\right]$.
Then $P$ is non singular with
$$P^{-1}=\left[\begin{array}{cccccccc}
e'&e_2&e_3&\cdots&e_{k}&e_{k+1}&\cdots&e_n
\end{array}\right],$$
where $e'$ is the column vector with the first entry equals to 1 and all other entries are $-1$.
Now,
\begin{eqnarray*}
    P^{-1}A_kP&=&P^{-1}\left[\begin{array}{cc}
        \gamma&a(k)^T\\
        \gamma j_{n-1}&A(k|k)\end{array}\right]\\
    &=&\left[\begin{array}{cc}
        \gamma&a(k)^T\\
        0&A(k|k)-j_{n-1}a(k)\end{array}\right],
\end{eqnarray*}
for $k=1,2\ldots,n.$ This shows that the matrices $A$ and $ \left[\begin{array}{cc}
        \gamma&a(k)^T\\
        0&A(k|k)-j_{n-1}a(k)\end{array}\right]$ are similar.
\end{proof}

Using  Theorem \ref{sub2} and the Ger{\v s}gorin circle theorem, next we establish a localization theorem for eigenvalues of a complex matrix.
\begin{theorem}
        \label{gers}
        Let $A=[a_{ij}]$ be an $n \times n$ complex  matrix  with the constant row sum $\gamma$. Then the eigenvalues of $A$ lie in the region
        $$\bigcap_{i=1}^n \Big{[}G_{A(i)}\cup\{\gamma\}\Big{]},$$
        where
        $G_{A(i)}=\bigcup_{k\neq i}G_{A(i)}(k)$ , with $G_{A(i)}(k)=\{z\in\mathbb{C}:|z-a_{kk}+a_{ik}|\leq\sum_{j\neq k}|a_{kj}-a_{ij}|\}.$
    \end{theorem}
\begin{proof}
By Ger{\v s}gorin circle theorem, each eigenvalue of $A(i),\ i=1,2,\ldots,n$ lies in the region
$G_A(i)$. Now, using Theorem \ref{sub2}, we get
\begin{eqnarray*}
G_{A(i)}&=&\bigcup_{\substack{k=1\\k\neq i}}^n\{z\in\mathbb{C}:|z-A(i)_{kk}|\leq\sum_{j\neq k}|A(i)_{kj}|\},\\
&=&\bigcup_{\substack{k=1\\k\neq i}}^n\{z\in\mathbb{C}|z-a_{kk}+a_{ik}|\leq\sum_{j\neq k}|a_{kj}-a_{ij}|\}.
\end{eqnarray*}
Since $A$ is similar to the matrix
$\left[\begin{array}{cc}
\gamma&a(i)^T\\
0&A(i)\end{array}\right],\ i=1,2,\ldots,n,$ all eigenvalues of $A$ lie in the region $G_A(i)\cup\{\gamma\}$, for all $i=1,2,\ldots,n.$ Hence
the eigenvalues of $A$ lie in the region
$$\bigcap_{i=1}^n \Big{[}G_{A(i)}\cup\{\gamma\}\Big{]}.$$
\end{proof}

\begin{remark}
For the following example the eigenvalue region obtained in the above theorem is smaller than the eigenvalue region obtained by Ger{\v s}gorin circle theorem. Consider the matrix,
$$A=\left[\begin{array}{ccc}
1&1+i&i\\
i&2+i&0\\
2&i&i
\end{array}\right]$$
By Ger{\v s}gorin circle theorem, all the eigenvalues of $A$ lie in the union of three Ger{\v s}gorin discs, say, $\sigma_1(A)$, $\sigma_2(A)$ and $\sigma_3(A)$, i.e.,
$$G_A=\sigma_1(A)\cup\sigma_2(A)\cup\sigma_3(A),$$
where
$$\sigma_1(A)=\{z\in\mathbb{C}:|z-1|\leq1+\sqrt{2}\},$$
$$\sigma_2(A)=\{z\in\mathbb{C}:|z-2-i|\leq1\},$$
and
$$\sigma_3(A)=\{z\in\mathbb{C}:|z-i|\leq3\}.$$
Now using Theorem \ref{sub2}, we have
$$A(1)=\left[\begin{array}{cc}
1&-i\\
-1&0\end{array}\right]$$
Let $G_{A(1)}=\sigma_1(A(1))\cup\sigma_2(A(1))$, where
$$\sigma_1(A(1))=\{z\in\mathbb{C}:|z-1|\leq1\},$$
and
$$\sigma_2(A(1))=\{z\in\mathbb{C}:|z|\leq1\}.$$
It is clear that,  $G_{A(1)}$ is a proper subset of $G_A$.
Hence $\bigcap_{i=1}^3 \Big{[}G_{A(i)}\cup\{2+2i\}\Big{]}$ is properly contained in $G_A$.
\end{remark}

In the following theorem, using Theorem \ref{sub2} and Theorem \ref{Bau}, we establish   a localization theorem for eigenvalues of a complex matrix.
\begin{theorem}
    \label{baur1main}
    Let $A=[a_{ij}]$ be an  $n \times n$ complex  matrix with the constant row sum $\gamma$. Then the eigenvalues of $A$ lie in the region
    $$\bigcap_{i=1}^n \Big{[}K_{A(i)}\cup\{\gamma\}\Big{]},$$
    where $K_{A(i)}$ is given by
    $$\bigcup_{\substack{j,k\neq i\\j\neq k}}\Big\{z\in\mathbb{C}:|z-a_{jj}+a_{ij}||z-a_{kk}+a_{ik}|\leq \Big(\sum_{l\neq j}|a_{jl}-a_{il}|\Big)\Big(\sum_{m\neq k}|a_{km}-a_{im}|\Big)\Big\}.$$
\end{theorem}
\begin{proof}
    By Theorem \ref{Bau}, each eigenvalue of $A(i),\ i=1,2,\ldots,n$ lies in the region
    $K_{A(i)}$. Now, using Theorem \ref{sub2}, we get
    \begin{eqnarray*}
        K_{A(i)}&=&\bigcup_{\substack{j,k\in I\setminus\{i\}\\j\neq k}}\Big\{z\in\mathbb{C}:|z-A(i)_{jj}||z-A(i)_{kk}|\leq\big(\sum_{l\neq j}|A(i)_{jl}|\big)\big(\sum_{m\neq k}|A(i)_{km}|\big)\Big\},\\
        &=&\bigcup_{\substack{j,k\neq i\\j\neq k}}^n\Big\{z\in\mathbb{C}:|z-a_{jj}+a_{ij}||z-a_{kk}+a_{ik}|\leq \Big(\sum_{l\neq j}|a_{jl}-a_{il}|\Big)\Big(\sum_{m\neq k}|a_{km}-a_{im}|\Big)\Big\}.
    \end{eqnarray*}
Since $A$ is similar to the matrix $\left[\begin{array}{cc}
\gamma&a(i)^T\\
0&A(i)\end{array}\right],\ i=1,2,\ldots,n,$
 all eigenvalues of $A$ lie in the region $K_A(i)\cup\{\gamma\}$, for all $i=1,2,\ldots,n.$ Hence
the eigenvalues of $A$ lie in the region
$$\bigcap_{i=1}^n \Big{[}K_{A(i)}\cup\{\gamma\}\Big{]}.$$
\end{proof}
\section{Eigenvalue bounds for the adjacency matrix of regular graphs}\label{Sec3}
For a graph $G$, let $A$ denote the adjacency matrix of $G$. Let $\lambda_1 \geq \lambda_2 \geq \dots \lambda_n$ be the eigenvalues of $A$. In this section, we derive bounds for the eigenvalues of the adjacency matrix of $d$-regular graphs.
\begin{theorem}
    \label{adj_th1}
    Let $G$ be a connected $d$-regular graph on $n$ vertices. Then
        $$-\frac{d}{n-1}-\frac{1}{n-1}\sqrt{(n-2)[nd(n-d-1)]}\leq\lambda_n\leq -\frac{d}{n-1}-\frac{1}{n-1}\sqrt{\frac{nd(n-d-1)}{n-2}}$$
    and
    $$-\frac{d}{n-1}+\frac{1}{n-1}\sqrt{\frac{nd(n-d-1)}{n-2}}\leq\lambda_2\leq-\frac{d}{n-1}+\frac{1}{n-1}\sqrt{(n-2)[nd(n-d-1)]}.$$
\end{theorem}
\begin{proof}
By Theorem \ref{sub1}, all the eigenvalues of the adjacency matrix of $G$ other than $d$ are also eigenvalues of  $A(k)$ for $k=1,2,\ldots,n$. Now for any $k \in \{1,\ldots, n\}$,
$$trace\ A(k)=-\lambda_1=-d,$$and
$$trace\  A(k)^2=\sum \lambda_i^2-\lambda_1^2 =nd-d^2.$$
Now, using Theorem \ref{tr1}, we establish the required bounds. Since $A(k)$ is a matrix of order $n-1$, we have
$$m=\frac{\text{trace } A(k)}{n-1}=-\frac{d}{n-1}\ \ \text{ and }\ \ s^2=\frac{\text{trace } A(k)^2}{n-1}-m^2=\frac{nd(n-d-1)}{(n-1)^2}. $$
Therefore, by Theorem \ref{tr1},
    $$-\frac{d}{n-1}-\frac{1}{n-1}\sqrt{(n-2)[nd(n-d-1)]}\leq\lambda_n\leq -\frac{d}{n-1}-\frac{1}{n-1}\sqrt{\frac{nd(n-d-1)}{n-2}},$$
    and
    $$-\frac{d}{n-1}+\frac{1}{n-1}\sqrt{\frac{nd(n-d-1)}{n-2}}\leq\lambda_2\leq-\frac{d}{n-1}+\frac{1}{n-1}\sqrt{(n-2)[nd(n-d-1)]}.$$
\end{proof}

The following result is known for the Laplacian matrix of connected graph.
\begin{theorem}\cite{Das2}
    \label{lap_th2}
    Let G be a simple connected graph with $n\geq 3$ vertices and let $0=\lambda_n<\lambda_{n-1}\leq\cdots \leq \lambda_1 $ be the Laplacian eigenvalues of  $G$. Then
    \begin{itemize}
        \item[(i)] $\lambda_{n-1}=\lambda_{n-2}=\cdots=\lambda_2$ if and only if G is a complete graph or a star graph or a regular complete bipartite graph.
        \item[(ii)] $\lambda_{n-2}=\cdots=\lambda_{2}=\lambda_1$ if and only if G is a complete graph or a graph $K_n-e$, where e is any edge.
    \end{itemize}
\end{theorem}

\begin{remark}
    It is well known that, a connected graph $G$ has two distinct eigenvalues if and only if $G$ is  complete, and a connected $d$-regular graph has three distinct eigenvalues if and only if the graph is  strongly regular\cite{Brou}. Now, for $d$-regular graphs, we have $A=dI_n-L$. For a connected graph $G$, using Theorem \ref{lap_th2},  we can observe that the following are equivalent:
\begin{enumerate}
\item  $G$ is $d$-regular and $G$ has three distinct eigenvalues with multiplicity of one eigenvalue is $n-2$,
\item  $G$ is strongly regular with degree of each vertex is $d$ and multiplicity of one eigenvalue is $n-2$,
\item  $G$ is $d$-regular complete bipartite.
\end{enumerate}
Now considering the equality conditions in Theorem \ref{tr1} we observe that, for equality in the inequalities of Theorem \ref{adj_th1}, $A$ must have at most three distinct eigenvalues with multiplicity of one eigenvalue is at least $n-2$. Thus,  using Theorem \ref{tr1} and the above argument, we conclude that\begin{itemize}
    \item[(i)] Equality hold in the right hand side of the inequalities in Theorem \ref{adj_th1} holds if and only if the graph is complete.
    \item[(ii)] Equality hold in the left hand side of the inequalities in Theorem \ref{adj_th1} holds if and only if the graph is complete  or regular complete bipartite.
    \end{itemize}
    \end{remark}

A bipartite graph $G$ is said to be $(c,d)$-biregular, if all vertices on one side of the bipartition have degree $c$ and all vertices on the other side have degree $d$.
In the next theorem, we derive bounds for the second largest eigenvalue of the $(c, d)$-biregular bipartite graphs.

\begin{theorem}
\label{Bip1}
Let $G$ be a connected bipartite graph on $n(\geq4)$ vertices whose vertex set is partitioned as $V=(X,Y)$. Let $d_i=c$ for all $i\in X$ and $d_i=d$ for all $i\in Y$, then
$$\sqrt{\frac{2(|E|-cd)}{(n-2)(n-3)}}\leq\lambda_2\leq\sqrt{\frac{2(n-3)(|E|-cd)}{n-2}}.$$
\end{theorem}
\begin{proof}
Since $G$ is $(c,d)$-biregular bipartite, we have $\lambda_1=\sqrt{cd}$ and $\lambda_n=-\sqrt{cd}.$
Let $B$ be a square matrix of order $n-2$ with eigenvalues $\mu_{n-2}\leq\mu_{n-3}\leq\ldots\leq\mu_1$ such that $\mu_i\in\sigma(A)\setminus\{\pm\sqrt{cd}\}$. Then $\mu_1=\lambda_2=-\lambda_{n-1}=-\mu_{n-2}$. Now, by applying Theorem \ref{tr1}, we get
$$m=\frac{\text{trace }B}{n-2}=\frac{\text{trace }A}{n-2}=0,$$
and
$$s^2=\frac{\text{trace }B^2}{n-2}-m^2=\frac{\text{trace }A^2-(\lambda_1^2+\lambda_n^2)}{n-2}=\frac{2|E|-2cd}{n-2}.$$
Therefore,
$$\sqrt{\frac{(2|E|-2cd)}{(n-2)(n-3)}}\leq\mu_1\leq\sqrt{\frac{(n-3)(2|E|-2cd)}{n-2}}.$$
\end{proof}
\begin{remark}
    \label{adj_rem2}
To prove Theorem \ref{Bip1}, we applied Theorem \ref{tr1} to a square matrix $B$ of order $n-2$. Now using the equality conditions of Theorem \ref{tr1}, we conclude that equality in Theorem \ref{Bip1} holds if and only if $G$ has an eigenvalue other than $\pm\sqrt{cd}$ with multiplicity at least $n-3$. Now since the eigenvalues of bipartite graphs are symmetric about $0$, we conclude that the equality in both side of Theorem \ref{Bip1} holds if and only if $G=K_{c,d}$.
\end{remark}
\begin{corollary}
Let $G$ be a connected $d$-regular bipartite graph on $n(\geq4)$ vertices. Then
$$\sqrt{\frac{d(n-2d)}{(n-2)(n-3)}}\leq\lambda_2\leq\sqrt{\frac{d(n-3)(n-2d)}{n-2}}.$$
\end{corollary}
\begin{proof}
Since $G$ is $d$-regular bipartite graph on $n$ vertices, we have $$|E|=\frac{nd}{2}.$$
Now the result follows from Theorem \ref{Bip1} by taking $c=d$.
\end{proof}

In the next theorem, we establish bounds for the second largest and the smallest eigenvalue of any $d$-regular graph in terms of the number of common neighbors of its vertices using Theorem \ref{gers}.
\begin{theorem}
    \label{adj_th2}
Let $G$ be a connected $d$-regular graph on $n$ vertices. Then
$$-2d+\max_{i\in G}\{\min_{k\neq i}\{\alpha_{ik}\},d\}\leq \lambda_n\leq\lambda_2\leq 2d-\max_{i\in G}\{\min_{k\neq i}\{\beta_{ik}\},d\},$$
where, for $k\neq i$, $\alpha_{ik}$ and $\beta_{ik}$ are given by
$$\alpha_{ik}=\begin{cases}
1+2N(i,k),&\textit{ if }k\sim i\\
2N(i,k),&\textit{ if }k\nsim i
\end{cases}$$
and
$$\beta_{ik}=\begin{cases}
3+2N(i,k),&\textit{ if }k\sim i\\
2N(i,k),&\textit{ if }k\nsim i.
\end{cases}$$
\end{theorem}
\begin{proof}
Let $\lambda$ be any eigenvalue of $A$ other than $d$. Then
by Theorem \ref{sub1}, $\lambda$  is also an eigenvalue of $A(i)=A(i|i)-\textbf{\textit{j}}_{n-1}\textbf{a}(i)^T$, where $a(i)^T=\left[\begin{array}{cccccc}
a_{i1}&\cdots&a_{i,i-1}&a_{i,i+1}&\cdots&a_{in}
\end{array}\right]$, for $i=1,2,\ldots,n$. So, by Theorem \ref{gers}, $\lambda$ lies in the region $G_{A(i)}$ with $$G_{A(i)}=\bigcup_{\substack{k=1\\k\neq i}}^n\Big{\{}z\in\mathbb{C}:|z+a_{ik}|\leq\sum_{\substack{j\neq k\\j\neq i}}|a_{kj}-a_{ij}|\Big{\}}=\bigcup_{\substack{k=1\\k\neq i}}^nG_{A(i)}(k).$$
 Now, for the vertex $k\in G$, $k\neq i$, we calculate the center and the radius of $G_{A(i)}(k)$ :\\
\textbf{Case I: }Let $k\sim i$. Then the disc $G_{A(i)}(k)$ is given by
\begin{eqnarray*}
    |z+1|&\leq&\sum_{j\neq i,k}|a_{kj}-a_{ij}|\\
    &=&\sum_{\substack{j\sim i,\\j\sim k}}|a_{kj}-a_{ij}|+\sum_{\substack{j\nsim i,\\j\sim k}}|a_{kj}-a_{ij}|+ \sum_{\substack{j\sim i,\\j\nsim k}}|a_{kj}-a_{ij}|+\sum_{\substack{j\nsim i,\\j\nsim k}}|a_{kj}-a_{ij}|\\
    &=&0+d-N(i,k)-1+d-N(i,k)-1+0\\
    &=&2d-2N(i,k)-2.
\end{eqnarray*}
\textbf{Case II: }Let $k\nsim i$. Then $a_{ik}=0$ and $a_{ki}=0$. Thus, we have
\begin{eqnarray*}
    |z|&\leq&\sum_{j\neq i,k}|a_{kj}-a_{ij}|\\
    &=&\sum_{\substack{j\sim i,\\j\sim k}}|a_{kj}-a_{ij}|+\sum_{\substack{j\nsim i,\\j\sim k}}|a_{kj}-a_{ij}|+ \sum_{\substack{j\sim i,\\j\nsim k}}|a_{kj}-a_{ij}|+\sum_{\substack{j\nsim i,\\j\nsim k}}|a_{kj}-a_{ij}|\\
    &=&0+d-N(i,k)+d-N(i,k)+0\\
    &=&2d-2N(i,k).
\end{eqnarray*}
Now, by combining  Case I and Case II, we can conclude that any eigenvalue $\lambda\neq d$ of $A$ must satisfy
$$-2d+\min_{k\neq i}\{\alpha_{ik}\}\leq \lambda\leq 2d-\min_{k\neq i}\{\beta_{ik}\},$$
for all $i=1,2,\ldots,n.$\\
Since $d$ is the spectral radius of $A$ and the above result is true for all $i=1,2,\ldots,n$, we have,
$$-2d+\max_{i\in G}\{\min_{k\neq i}\{\alpha_{ik}\},d\}\leq \lambda_n\leq\lambda_2\leq 2d-\max_{i\in G}\{\min_{k\neq i}\{\beta_{ik}\},d\}.$$
\end{proof}

\begin{remark}
Let $G=K_n$, the complete graph on $n$ vertices, then $d_i=n-1$ for all $i\in G$ and $N(i,k)=n-2$ for all $i,k\in G$. Therefore $\alpha_{ik}=2n-3$ and $\beta_{ik}=2n-1$. By Theorem \ref{adj_th2}, we have $$-1\leq\lambda_n\leq\lambda_2\leq-1.$$
Thus the bounds obtained in Theorem \ref{adj_th2} are sharp for the complete graphs.
\end{remark}

In the next theorem, we establish bounds for the second largest and the smallest eigenvalue of any $d$-regular graph in terms of the number of common neighbors of its vertices using Theorem \ref{baur1main}.
\begin{theorem}
    \label{adj_th3}
    Let $G$ be a connected $d$-regular graph on $n$ vertices. Then
    $$\max_{i\in G}\min_{\substack{j,k\in I\setminus \{i\}\\j\neq k}}\{\alpha_{ijk}\}\leq \lambda_n\leq\lambda_2\leq\min_{i\in G}\max_{\substack{j,k\in I\setminus \{i\}\\j\neq k}}\{\beta_{ijk}\}.$$
    where, for $j,k\neq I\setminus \{i\}$, $j\neq{k}$,
    $$\alpha_{ijk}=\begin{cases}
        -1-2\sqrt{(d-N(i,j)-1)(d-N(i,k)-1)},&\text{ if }j\sim i,\ k\sim i,\\
    -2\sqrt{(d-N(i,j))(d-N(i,k))},&\text{ if }j\nsim i,\ k\nsim i,\\
    -\frac{1}{2}-\sqrt{\frac{1}{4}+4(d-N(i,j)-1)(d-N(i,k))},&\text{ if }j\sim i,\ k\nsim i
    \end{cases}$$
    and
    $$\beta_{ijk}=\begin{cases}
            -1+2\sqrt{(d-N(i,j)-1)(d-N(i,k)-1)},&\text{ if }j\sim k,\ j\sim i,\\
        2\sqrt{(d-N(i,j))(d-N(i,k))},&\text{ if }j\nsim i,\ k\nsim i,\\
        -\frac{1}{2}+\sqrt{\frac{1}{4}+4(d-N(i,j)-1)(d-N(i,k))},&\text{ if }j\sim i,\ k\nsim i.
    \end{cases}$$
\end{theorem}
\begin{proof}
Since $G$ is a connected $d$ regular graph, each eigenvalue of $A$ other than $d$ is also an eigenvalue of  $A(i)$, for $i\in I$. Let $\lambda\neq d$ be any eigenvalue of $A$, then $$\lambda\in K_{A(i)}=\bigcup_{\substack{j,k\in I\setminus\{i\}\\j\neq k}}K_{ij}(A(i)),\ \forall i\in I,$$ where $K_{ij}(A(i)) =\bigcup_{\substack{k\in I\setminus\{i\}}}\Big\{z\in\mathbb{C}:|z-A(i)_{jj}||z-A(i)_{kk}|\leq\big(\sum_{l\neq j}|A(i)_{jl}|\big)\big(\sum_{m\neq k}|A(i)_{km}|\big)\Big\}.$
Now, for $j\neq i$, the $j$deleted absolute row sum of $A(i)$ is given by:
$$r_j(A(i))=\begin{cases}
2d-2N(i,j)-2,&\text{ if }i\sim j,\\
2d-2N(i,j),&\text{ if }i\nsim j.
\end{cases}$$
Let us compute the regions $K_{jk}(A(i))$. Here three cases arises:\\
\textbf{Case I.} If $j\sim i,\ k\sim i$. Then the region $K_{jk}(A(i))$ is given by
\begin{eqnarray*}
|z+1|^2&\leq&(2d-2N(i,j)-2)(2d-2N(i,k)-2)\\
&=&4(d-N(i,j)-1)(d-N(i,k)-1).
\end{eqnarray*}
\textbf{Case II.} If $j\nsim i,\ k\nsim i$.  Then the region $K_{jk}(A(i))$ is given by
\begin{eqnarray*}
|z|^2&\leq&(2d-2N(i,j))(2d-2N(i,k))\\
&=&4(d-N(i,j))(d-N(i,k)).
\end{eqnarray*}
\textbf{Case III.} If the vertex $i$ is adjacent to exactly one of the vertices $j$ and $k$. Let $i\sim j$ and $i\nsim k$. Then the region $K_{jk}(A(i))$ is given by
\begin{eqnarray*}
|z(z+1)|&\leq&(2d-2N(i,j)-2)(2d-2N(i,k))\\
&=&4(d-N(i,j)-1)(d-N(i,k)).
\end{eqnarray*}
This gives
$$|z+\frac{1}{2}|^2\leq\frac{1}{4}+4(d-N(i,j)-1)(d-N(i,k)).$$
Since the eigenvalues of $A$ are real,  combining all these cases, we get
$$\min_{\substack{j,k\in I\setminus \{i\}\\j\neq k}}\{\alpha_{ijk}\}\leq\lambda\leq\max_{\substack{j,k\in I\setminus \{i\}\\j\neq k}}\{\beta_{ijk}\},\ \forall i\in I.$$
Since the above result is true for $i=1,2,\ldots n.$ Therefore
$$\max_{i\in G}\min_{\substack{j,k\in I\setminus \{i\}\\j\neq k}}\{\alpha_{ijk}\}\leq\lambda\leq\min_{i\in G}\max_{\substack{j,k\in I\setminus \{i\}\\j\neq k}}\{\beta_{ijk}\}.$$
\end{proof}

\begin{remark}
If $G=K_n$ then $\alpha_{ijk}=-1$ and $\beta_{ijk}=-1$ for all  $j,k\in I\setminus\{i\},\ j\neq k$. Thus, when $G=K_n$, Theorem \ref{adj_th3} provides sharp bound for eigenvalues of $G$.
\end{remark}

\section{Eigenvalue bounds for normalized adjacency matrix}\label{Sec4}

%

Let $\mathcal{A}$ denote the normalized adjacency matrix of the graph $G$, and let $\lambda_1 \geq \lambda_2 \geq \dots \geq \lambda_n$ be the eigenvalues of $\mathcal{A}$. In this section, we establish bounds for the eigenvalues of the normalized adjacency matrix connected graphs.
\begin{theorem}
    \label{nadj_th1}
    Let $G$ be a connected graph on $n$ vertices. Then the smallest and the second largest eigenvalue of the normalized adjacency matrix satisfy
    $$-\frac{1}{n-1}-\frac{1}{n-1}\sqrt{(n-2)[2(n-1)R_{-1}(G)-n]}\leq\lambda_n\leq -\frac{1}{n-1}-\frac{1}{n-1}\sqrt{\frac{2(n-1)R_{-1}(G)-n}{n-2}},$$
    and
    $$-\frac{1}{n-1}+\frac{1}{n-1}\sqrt{\frac{2(n-1)R_{-1}(G)-n}{n-2}}\leq\lambda_2\leq-\frac{1}{n-1}+\frac{1}{n-1}\sqrt{(n-2)[2(n-1)R_{-1}(G)-n]}.$$
\end{theorem}
\begin{proof}
By Theorem \ref{sub1}, any eigenvalue of $\mathcal{A}$ other than $1$ is also an eigenvalue of $\mathcal{A}(k)$. Therefore
$$\text{trace } \mathcal{A}(k)=\text{trace } \mathcal{A}-1\ \text{ and }\ \text{trace } \mathcal{A}(k)^2=\text{trace } \mathcal{A}^{2} - 1,$$
for all $k=1,2,\ldots,n.$\\
Thus, for any $k$, we have
$$m=\frac{\text{trace } \mathcal{A}(k)}{n-1}=-\frac{1}{n-1}$$
and
$$s^2=\frac{\text{trace } \mathcal{A}(k)^2}{n-1}-m^2=\frac{\sum_{i\sim j} \frac{1}{d_id_j}-1}{n-1}-\frac{1}{(n-1)^2}=\frac{2(n-1)R_{-1}(G)-n}{(n-1)^2}. $$
Therefore, by Theorem \ref{tr1}, we have
$$-\frac{1}{n-1}-\frac{1}{n-1}\sqrt{(n-2)[2(n-1)R_{-1}(G)-n]}\leq\lambda_n\leq -\frac{1}{n-1}-\frac{1}{n-1}\sqrt{\frac{2(n-1)R_{-1}(G)-n}{n-2}}$$
and
$$-\frac{1}{n-1}+\frac{1}{n-1}\sqrt{\frac{2(n-1)R_{-1}(G)-n}{n-2}}\leq\lambda_2\leq-\frac{1}{n-1}+\frac{1}{n-1}\sqrt{(n-2)[2(n-1)R_{-1}(G)-n]}.$$
\end{proof}

\begin{remark}
If equality holds in Theorem \ref{nadj_th1}, then the graph $G$ must have at most three normalized adjacency eigenvalues such that multiplicity of one eigenvalue at least $n-2$. Now recall that, $G=K_n$  is the only graph with multiplicity of $\lambda_2(=\frac{-1}{n-1})$ is $n-2$(\ref{}), and $G=K_{p,q}$ is the only graph with three eigenvalues (see Remark 2.6.5, \cite{Mcav}) with multiplicity of the eigenvalue is $n-2$. Thus using the equality conditions of Theorem \ref{tr1}, we conclude that
 \begin{itemize}
    \item[(i)]  equality holds in the right hand side of both inequalities in Theorem \ref{nadj_th1}  if and only if the graph is a complete graph, and
    \item[(ii)] equality holds in the left hand side of both inequalities in Theorem \ref{nadj_th1}  if and only if the graph is a complete graph or a complete bipartite graph.
 \end{itemize}
\end{remark}
Now, let us establish  bounds for the eigenvalues of bipartite graphs.
\begin{theorem}
\label{nadj_th2}
Let $G$ be a connected bipartite graph on $n$ vertices. Then
$$\sqrt{\frac{2(R_{-1}(G)-1)}{(n-2)(n-3)}}\leq\lambda_2\leq\sqrt{\frac{2(n-3)(R_{-1}(G)-1)}{(n-2)}}.$$
\end{theorem}
\begin{proof}
Since $G$ is bipartite, we have $\lambda_n=-1.$
Let $B$ be a square matrix of order $n-2$ with eigenvalues $\mu_{n-2}\leq\mu_{n-3}\leq\ldots\leq\mu_1$ such that $\mu_i\in\sigma(\mathcal{A})\setminus\{\pm1\}$. Then $\mu_1=\lambda_2=-\lambda_{n-1}=-\mu_{n-2}$. By Theorem \ref{tr1}, we have
$$m=\frac{\text{trace }B}{n-2}=\frac{\text{trace }\mathcal{A}}{n-2}=0,$$
and
$$s^2=\frac{\text{trace }B^2}{n-2}-m^2=\frac{\text{trace }\mathcal{A}^2-(\lambda_1^2+\lambda_n^2)}{n-2}=\frac{2R_{-1}(G)-2}{n-2}.$$
Therefore,
$$\sqrt{\frac{(2R_{-1}(G)-2)}{(n-2)(n-3)}}\leq\mu_1\leq\sqrt{\frac{(n-3)(2R_{-1}(G)-2)}{n-2}}.$$
Hence, we get
$$\sqrt{\frac{2(R_{-1}(G)-1)}{(n-2)(n-3)}}\leq\lambda_2\leq\sqrt{\frac{2(n-3)(R_{-1}(G)-1)}{(n-2)}}.$$
\end{proof}

\begin{remark}
By a similar reason as in Remark \ref{adj_rem2}, we conclude that, the equality on both side of Theorem \ref{nadj_th2} holds if and only if the graph is a complete bipartite graph.
\end{remark}

\begin{theorem}
    \label{nadj_th3}
        Let $G$ be a simple connected graph of order n with dominating vertex. If $d_i=n-1$, for some $i\in G$, then
        $$-2-\frac{2}{n-1}+\min_{k\neq i}\big\{\frac{1}{d_k}+\frac{2d_k}{n-1}\big\}\leq \lambda_n\leq\lambda_2\leq 2-\min_{k\neq i}\big\{\frac{1}{d_k}+\frac{2d_k}{n-1}\big\}.$$
    \end{theorem}
\begin{proof}
Since the vertex $i$ is a dominating vertex, therefore $a_{ij}=\frac{1}{n-1}$ for all $j\neq i$, we have
\begin{eqnarray*}
    r_k(\mathcal{A}(i))&=&\sum_{j\in I\setminus\{i,k\}}|\mathcal{A}(i)_{kj}|\\
    &=&\sum_{j\in I\setminus\{i,k\}}|a_{jk}-\frac{1}{n-1}|\\
    &=&\sum_{\substack{j\sim k,\\j\neq i}}|a_{jk}-\frac{1}{n-1}|+\sum_{\substack{j\nsim k,\\j\neq i}}|a_{jk}-\frac{1}{n-1}|\\
    &=&(\frac{1}{d_k}-\frac{1}{n-1})(d_k-1)+\frac{1}{n-1}(n-d_k-1)\\
    &=&2-\frac{1}{d_k}-\frac{2d_k-1}{n-1}.
\end{eqnarray*}
Let $\lambda$ be any eigenvalue of $\mathcal{A}(i)$, then, by Ger{\v s}gorin circle theorem, there exists $k\neq i$ such that $\lambda$ satisfies
$$|\lambda+\frac{1}{n-1}|\leq 2-\frac{1}{d_k}-\frac{2d_k-1}{n-1}.$$
Since, the eigenvalues of $\mathcal{A}$ are real and each eigenvalue of $\mathcal{A}$ other than 1 is also an eigenvalue of $\mathcal{A}(i)$. Thus the required inequality follows by considering all possibilities in the above inequality.
\end{proof}
\begin{remark}
The bounds in Theorem \ref{nadj_th3} are sharp for complete graph.
\end{remark}

\begin{theorem}
    \label{nadj_th4}
        Let $G$ be a simple connected graph with a dominating vertex $i$. Then
        $$\lambda_n\geq-\frac{1}{n-1}-\max_{\substack{j,k\in I\setminus \{i\}\\j\neq k}}\sqrt{\Big(2-\frac{1}{d_j}-\frac{2d_j-1}{n-1}\Big)\Big(2-\frac{1}{d_k}-\frac{2d_k-1}{n-1}\Big)} ,$$
        and
        $$\lambda_2\leq-\frac{1}{n-1}+\max_{\substack{j,k\in I\setminus \{i\}\\j\neq k}}\sqrt{\Big(2-\frac{1}{d_j}-\frac{2d_j-1}{n-1}\Big)\Big(2-\frac{1}{d_k}-\frac{2d_k-1}{n-1}\Big)}.$$
    \end{theorem}
\begin{proof}
We use Theorem \ref{Bau} to $\mathcal{A}(i)$ to establish above inequalities. Since $a_{ij}=\frac{1}{n-1}$ for all $j\neq i$, we have
\begin{eqnarray*}
r_k(\mathcal{A}(i))=2-\frac{1}{d_k}-\frac{2d_k-1}{n-1}
\end{eqnarray*}
Let $\lambda$ be an eigenvalue of $\mathcal{A}$ other than 1. Then $\lambda$ must satisfy
$$|\lambda-\frac{1}{n-1}|^2\leq\Big(2-\frac{1}{d_j}-\frac{2d_j-1}{n-1}\Big)\Big(2-\frac{1}{d_k}-\frac{2d_k-1}{n-1}\Big)$$
for some $j,k\in I\setminus\{i\},\ j\neq k.$\\
Now since all the eigenvalues of $\mathcal{A}$ are real, the result follows from the above inequality.
\end{proof}

\begin{remark}
Equality on both side of Theorem \ref{nadj_th4} holds if the graph is a complete graph.
\end{remark}

\section{Eigenvalue bounds for Laplacian matrix}\label{Sec5}
In this section, we establish bounds for eigenvalues of Laplacian matrices of connected graphs. First, we observe the following result about the nonzero eigenvalues of the Laplacian matrix of connected graphs.
\begin{theorem}
    \label{La1}
Let $G$ be a connected graph with $n$ vertices. Then any non-zero eigenvalue of $L=[l_{ij}]$ is also an eigenvalue of the matrix
$$L(k)=L(k|k)-j_{n-1}l(k),\ \ k=1,2,\ldots, n,$$
where $l(k)^T=\left[\begin{array}{cccccc}
l_{k1}&\cdots&l_{k,k-1}&l_{k,k+1}&\cdots&l_{kn}
\end{array}\right].$
\end{theorem}
\begin{proof}
Since $G$ is connected, so $0$ is a simple eigenvalue of $L$. Therefore the result follows from Theorem \ref{sub2}.
\end{proof}
Using the above result and Theorem \ref{tr1}, let us derive bounds for the largest and the second smallest eigenvalues of Laplacian matrix of  a connected graph.
\begin{theorem}
    \label{lap_th1}
    Let $G$ be a connected graph on $n$ vertices. Then the second smallest eigenvalue $\lambda_{n-1}$ and the largest eigenvalue $\lambda_1$ of the Laplacian matrix satisfy
    $$\frac{n}{n-1}\Delta-\sqrt{\frac{n-2}{n-1}\Bigg[\sum d_i^2+n\Delta-\frac{n^2\Delta^2}{n-1}\Bigg]}\leq\lambda_{n-1}\leq \frac{n}{n-1}\Delta-\sqrt{\frac{\sum d_i^2+n\Delta-\frac{n^2\Delta^2}{n-1}}{(n-1)(n-2)}}$$
    and
    $$\frac{n}{n-1}\Delta+\sqrt{\frac{\sum d_i^2+n\Delta-\frac{n^2\Delta^2}{n-1}}{(n-1)(n-2)}}\leq\lambda_1\leq\frac{n}{n-1}\Delta+\sqrt{\frac{n-2}{n-1}\Bigg[\sum d_i^2+n\Delta-\frac{n^2\Delta^2}{n-1}\Bigg]}.$$
\end{theorem}
\begin{proof}
Since any nonzero eigenvalue of $L$ is also an eigenvalue of $L(k)$. Therefore
$$\text{trace } L(k)=\text{trace } L\ \text{ and }\ \text{trace } L(k)^2=\text{trace } L^2,$$
for all $k=1,2,\ldots,n.$\\
Thus, for any $k$, we have
$$m=\frac{\text{trace } L(k)}{n-1}=\frac{\sum d_i}{n-1}=\frac{n}{n-1}\Delta$$
and
$$s^2=\frac{\text{trace } A(k)^2}{n-1}-m^2=\frac{\sum d_i^2-\sum d_i}{n-1}-\Delta^2=\frac{\sum d_i}{n-1}-\frac{n}{n-1}\Delta-\dfrac{n^2}{(n-1)^2}\Delta^2. $$
The required result can be obtained by using these in Theorem \ref{tr1}.
\end{proof}
\begin{remark}
If equality holds in Theorem \ref{lap_th1}, then the graph $G$ must have at most three Laplacian eigenvalues with multiplicity of one eigenvalue at least $n-2$. Thus, using Theorem \ref{tr1} and Theorem \ref{lap_th2}, we conclude that:
  \begin{itemize}
    \item[(i)] equality holds in the right hand side of both inequalities in Theorem \ref{nadj_th1} if and only if the graph is a complete graph or a star graph or a regular complete bipartite graph, and
    \item[(ii)] equality holds in the left hand side of both inequalities in Theorem \ref{nadj_th1} holds if and only if the graph is a complete graph or a graph $K_n-e$, graph obtained by deleting an edge $e$ from the complete graph.
 \end{itemize}
\end{remark}
Next, we derive  bounds for the eigenvalues of connected graphs  in terms of the number of common neighbors of its vertices using Theorem \ref{La1}.
\begin{theorem}
\label{la_th2}
Let $G$ be a connected graph on $n$ vertices. Then the largest eigenvalue $\lambda_1$ and the second smallest eigenvalue $\lambda_{n-1}$ of the Laplacian matrix satisfy
$$\max_{i\in G}\min_{k\neq i}\alpha_{ik}\leq \lambda_{n-1}\leq\lambda_1\leq\min_{i\in G}\max_{k\neq i}\beta_{ik},$$
where, for $k\neq i$, $\alpha_{ik}$ and $\beta_{ik}$ are given by
$$\alpha_{ik}=\begin{cases}
-d_i+2N(i,k)+1,&\textit{ if }k\sim i,\\
-d_i+2N(i,k),&\textit{ if }k\nsim i,
\end{cases}$$
and
$$\beta_{ik}=\begin{cases}
2d_i+d_k-2N(i,k)-1,&\textit{ if }k\sim i,\\
2d_i+d_k-2N(i,k),&\textit{ if }k\nsim i.
\end{cases}$$
\end{theorem}
\begin{proof} By Theorem \ref{La1}, each nonzero eigenvalue of $L$ is also an eigenvalue of $L(i)$, for $i=1,2,\ldots,n.$ So, by Theorem \ref{gers}, $\lambda$ lies in the regions $G_{L(i)}$ with $$G_{L(i)}=\bigcup_{k\neq i}\Big{\{}z\in\mathbb{C}:|z-l_{kk}+l_{ik}|\leq\sum_{\substack{j\neq k\\j\neq i}}|l_{kj}-l_{ij}|\Big{\}}=\bigcup_{\substack{k=1\\k\neq i}}^nG_{L(i)}(k).$$

For the vertex $k\in G$, $k\neq i$, let us calculate the regions $G_{L(i)}(k)$:\\
    \textbf{Case I:} If $k\sim i$, then we have
\begin{eqnarray*}
    |z-d_k-1|&\leq&\sum_{j\neq i,k}|l_{kj}-l_{ij}|\\
    &=&\sum_{\substack{j\sim i,\\j\sim k}}|l_{kj}-l_{ij}|+\sum_{\substack{j\nsim i,\\j\sim k}}|l_{kj}-l_{ij}|+ \sum_{\substack{j\sim i,\\j\nsim k}}|l_{kj}-l_{ij}|+\sum_{\substack{j\nsim i,\\j\nsim k}}|l_{kj}-l_{ij}|\\
    &=&0+d_k-N(i,k)-1+d_i-N(i,k)-1+0\\
    &=&d_i+d_k-2N(i,k)-2.
\end{eqnarray*}
\textbf{Case II: }If $k\nsim i$, then $l_{ik}=0$ and $l_{ki}=0$. Thus, we have
\begin{eqnarray*}
    |z-d_k|&\leq&\sum_{j\neq i,k}|l_{kj}-l_{ij}|\\
    &=&\sum_{\substack{j\sim i,\\j\sim k}}|l_{kj}-l_{ij}|+\sum_{\substack{j\nsim i,\\j\sim k}}|l_{kj}-l_{ij}|+ \sum_{\substack{j\sim i,\\j\nsim k}}|l_{kj}-l_{ij}|+\sum_{\substack{j\nsim i,\\j\nsim k}}|l_{kj}-l_{ij}|\\
    &=&0+d_k-N(i,k)+d_i-N(i,k)+0\\
    &=&d_i+d_k-2N(i,k).
\end{eqnarray*}
Now, since the eigenvalues of $L$ are real, from Case I and Case II we have
$$-d_i+2N(i,k)+1\leq\lambda\leq d_i+2d_k-2N(i,k)-1,\text{ if }i\sim k,$$
and
$$-d_i+2N(i,k)\leq\lambda\leq d_i+2d_k-2N(i,k),\text{ if }i\nsim k.$$
Now,  by considering all possible discs, we get
$$\min_{k\neq i}\alpha_{ik}\leq\lambda\leq\max_{k\neq i}\beta_{ik}.$$
Since the above inequality is true for all $i=1,2,\ldots n$, therefore
$$\max_{i\in G}\min_{k\neq i}\alpha_{ik}\leq\lambda\leq\min_{i\in G}\max_{k\neq i}\beta_{ik}.$$
\end{proof}
\begin{remark}
The bounds in Theorem \ref{lap_th2} are sharp for complete graphs.
\end{remark}
Next, we derive  bounds for the eigenvalues of connected graphs  in terms of the degrees of its vertices.
\begin{theorem}
\label{}
Let $G$ be a simple connected graph of order n with dominating vertex. If $d_i=n-1$, for some $i\in G$, then  $$\lambda_{n-1}\geq \frac{1}{2}\max_{\substack{j,k\in I\setminus \{i\}\\j\neq k}}\big\{d_j+d_k+2-\sqrt{(d_j-d_k)^2+4(n-d_j)(n-d_k)}\big\},$$
and
$$\lambda_1\leq \frac{1}{2}\max_{\substack{j,k\in I\setminus \{i\}\\j\neq k}}\big\{d_j+d_k+2+\sqrt{(d_j-d_k)^2+4(n-d_j)(n-d_k)}\big\}.$$
\end{theorem}
\begin{proof}
 Since $l_{ij}=-1$ for all $j\neq i$, we have
\begin{eqnarray*}
    r_k(L(i))&=&\sum_{j\neq k}|L(i)_{kj}|\\
    &=&\sum_{\substack{j\nsim k,\\j\sim i}}1\\
    &=&n-d_{k}
\end{eqnarray*}
Let $\lambda$ be any nonzero eigenvalue of $L$. Then $\lambda$ is also an eigenvalue of $L(i)$. By Theorem \ref{Bau}, there exists $j,k\in I\setminus \{i\},\ j\neq k$, such that $\lambda$ lies in the region:
$$|z-d_j-1|.|z-d_k-1|\leq (n-d_j)(n-d_k).$$
Since, eigenvalues of $L$ are real we have
$$|\lambda-\frac{d_j+d_k+2}{2}|^2\leq \frac{1}{4}(d_j-d_k)^2+(n-d_j)(n-d_k)
$$
This gives,
$$\lambda\leq\frac{d_j+d_k+2}{2}+\frac{1}{2}\sqrt{(d_j-d_k)^2+4(n-d_j)(n-d_k)},$$
and
$$\lambda\geq\frac{d_j+d_k+2}{2}-\frac{1}{2}\sqrt{(d_j-d_k)^2+4(n-d_j)(n-d_k)}.$$
Now considering all the possibilities in the above inequality, we get the required result.
\end{proof}
\section*{Acknowledgments}
Ranjit Mehatari is funded by NPDF (File no.- PDF/2017/001312), SERB, India. M. Rajesh Kannan would like to thank Department of Science and Technology, India, for the financial support(Earily Carrier Research Award(ECR/2017/000643)).

\end{document}